\documentclass[12pt,reqno,a4paper]{amsart}

\usepackage{amsfonts}
\usepackage{epsfig}
\usepackage{graphicx}
\usepackage{amsmath}
\usepackage{amssymb}
\usepackage{color}

\newcommand{\Qq}{\mathbb{Q}}

\newcommand{\NN}{\mathcal{N}}

\definecolor{cadmiumgreen}{rgb}{0.0, 0.42, 0.24}

\newcounter{main}

\numberwithin{equation}{section}

\newtheorem{theorem}{Theorem}[section]

\newtheorem{lemma}[theorem]{Lemma}
\newtheorem{corollary}[theorem]{Corollary}

\newtheorem{remark}{Remark}[section]
\newtheorem{definition}{Definition}[section]
\newtheorem{maintheorem}{Theorem}
\newtheorem{maincorollary}{Corollary}
\newtheorem{claim}{Claim}[section]

\newcommand{\blanksquare}{\,\,\,$\sqcup\!\!\!\!\sqcap$}

\newcounter{example}
{{\stepcounter{example}}{\flushleft {\bf Example \arabic{example}:}}}%
{\par}

\title[Hyperbolicity and stability for Hamiltonian flows]{Hyperbolicity and stability for Hamiltonian flows}

\date{\today}

\author[M. Bessa]{M\'{a}rio Bessa}

\address{Departamento de Matem\'atica, Universidade da Beira Interior, Rua Marqu\^es d'\'Avila e Bolama,
  6201-001 Covilh\~a,
Portugal.}
\email{bessa@fc.up.pt}

\author[J. Rocha]{Jorge Rocha}
\address{Departamento de Matem\'atica, Universidade do Porto, 
Rua do Campo Alegre, 687, 
4169-007 Porto, Portugal}
\email{jrocha@fc.up.pt}

\author[M. J. Torres]{Maria J. Torres}
\address{CMAT, Departamento de Matem\'atica e Aplica\c{c}\~{o}es, Universidade do Minho, 
Campus de Gualtar, 
4700-057 Braga, Portugal}
\email{jtorres@math.uminho.pt}

\begin{document}

\maketitle

\begin{abstract}

We prove that a Ha\-mil\-to\-nian star system, defined on a $2d$-dimensional symplectic manifold $M$ ($d\geq 2$), is Anosov. 
As a consequence we obtain the proof of the stability conjecture for Hamiltonians. This generalizes the $4$-dimensional results in ~\cite{BFR}.

\end{abstract}

\begin{section}{Introduction}

\begin{subsection}{Structural stability and hyperbolicity}

Let $\mathcal{S}$ be a dynamical system defined on a closed manifold.  The concept of {\em structural stability} was introduced in the mid 1930s by Andronov and Pontrjagin (\cite{AP}). Roughly speaking it means that under small perturbations the dynamics are 
topologically equivalent: a dynamical system is $C^r$-structurally stable
if it is topologically conjugated to any other system in a $C^r$ neighbourhood. These conjugations are often defined
in sets where the dynamics is relevant, usually in its non-wandering set, $\Omega(\mathcal{S})$, and the system is said to be {\em $\Omega$-stable}. We recall that $\Omega(\mathcal{S})$ is the set of points in the manifold such that, for every neighbourhood $U$, there exists an iterate $n$ sa\-tis\-fying $\mathcal{S}^n(U)\cap U\not=\emptyset$. Smale's program in the early $1960$s aimed to prove the (topological) genericity of structurally stable systems. Although Smale's program was proved to be wrong one decade later, it played a fundamental role in the development of the theory of dynamical systems. It led to the construction of Hyperbolic theory, studying uniform hyperbolicity, and characterizing structural stability as being essentially equivalent
to uniform hyperbolicity. In the attempt to unify several classes of structurally stable systems, e.g.,
Morse-Smale systems, the horseshoe and Anosov's systems, 
Smale conceived Axiom A: a system $\mathcal{S}$ is said to satisfy 
the \emph{Axiom A} property if the closure of its closed orbits is equal to $\Omega(\mathcal{S})$ and, moreover, this set is hyperbolic. 
It turned out to be one of the most challenging problems in the modern theory of dynamical systems to know if a $C^r$-structurally stable system satisfies the Axiom A property. A cornerstone to this program was the remarkable proof done by Ma\~n\'e of the stability conjecture for the case of $C^1$-dissipative diffeomorphisms (\cite{Ma1}). 
The proof of Ma\~n\'e essentially uses the property, that holds for stable diffeomorphisms, that all periodic orbits are robustly hyperbolic. 
Therefore, one could ask if there exists a weaker property than stability that guarantees Axiom A. This remounts to another old problem attributed to Liao and Ma\~n\'e (see, e.g.~\cite{GanWen}) that asks wether for a system to loose the $\Omega$-stability it must undergo a bifurcation in a critical element. In other words, must 
a system robustly free of any critical-element-bifurcation be $\Omega$-stable? 
\end{subsection}
\begin{subsection}{The star systems} Back to the early 1980s, Ma\~n\'e defined a set $\mathcal{F}^1$, of dissipative diffeomorphisms having a $C^1$-neighbourhood $\mathcal{U}$ such that every diffeomorphism inside $\mathcal{U}$ has all periodic orbits of hyperbolic type. Therefore, a diffeomorfism in $\mathcal{F}^1$ is a system that has robustly no critical-element-bifurcation. 
Given that being in $\mathcal{F}^1$
concerns only to critical points and that the hyperbolicity on critical points is merely orbit-wise, but not uniform, 
this property looks, a priori, quite weak. Indeed, the Axiom A plus the no-cycle property, which is necessary and sufficient for a system to be $\Omega$-stable looks much stronger. Recall that, by the spectral decomposition of an Axiom A system $\mathcal{S}$, we have that $\Omega(\mathcal{S})=\cup_{i=1}^k \Lambda_i$ where each $\Lambda_i$ is a basic piece. We define an order relation by $\Lambda_i\prec \Lambda_j$ if there exists $x$ (outside $\Lambda_i\cup\Lambda_j$) such that $\alpha(x)\subset \Lambda_i$ and $\omega(x)\subset\Lambda_j$. We say that $\mathcal{S}$ has a \emph{cycle} if there exists a cycle with respect to $\prec$ (see ~\cite{Shub} for details). Thus, the above conjecture of Liao and Ma\~n\'e can be stated as follows: does every system robustly free of non-hyperbolic critical elements satisfy Axiom A and the no-cycle property? For diffeomorphisms the answer is affirmative. 
In ~\cite{Ma2}, Ma\~n\'e proved that every surface dissipative diffeomorphism of $\mathcal{F}^1$ satisfies the Axiom A property. Hayashi (\cite{Hay}) extended this result for higher dimensions. 
In fact, the mentioned results by Ma\~n\'e and Hayashi guarantee that diffeomorphisms in $\mathcal{F}^1$ satisfy the Axiom A and the no-cycle properties (see also a result by Aoki~\cite{Aoki}). We point out that classic results imply that being in $\mathcal{F}^1$ is a necessary condition to satisfy the Axiom A and the no-cycle condition (see~\cite{Ma1} and the references wherein). 
In the conservative setting 
we refer the seminal paper of Newhouse~\cite{N} where it was proved that
any symplectomorphism robustly free of non-hyperbolic periodic orbits is Anosov. 
Recently, an analogous result was obtained by Arbieto and Catalan~\cite{AC} for volume-preserving diffeomorphisms.

For the continuous-time case the analogous to the set $\mathcal{F}^1$ is traditionally denoted by $\mathcal{G}^1$, and a flow in it is called a {\em star flow}. Obviously, in this setting, the hyperbolicity of the equilibrium points (singularities of the vector field) is also imposed.

It is well known that the dissipative star flow defined by the Lorenz differential equations (see e.g. ~\cite{WT})  belongs to $\mathcal{G}^1$. However, the hyperbolic saddle-type singularity is accumulated by (hyperbolic) closed orbits and they are contained in the non-wandering set preventing the flow to be Axiom A. The problem of Liao and Ma\~n\'e of knowing if every (nonsingular) dissipative star flow satisfies the Axiom A and the no-cycle condition remained unsolved for almost 20 years, in part due to the technical difficulties specific of the flow setting. This central result was proved by Gan and Wen (\cite{GanWen}).

If we consider flows that are divergence-free and restrict the definition of $\mathcal{G}^1$ to this setting, which means that the star property is satisfied when one restricts to the conservative setting (but possibly not in the broader space of dissipative flows), using a completely different approach, based in conservative-type seminal ideas of Ma\~n\'e, two of the authors (see~\cite{MBJR}) proved recently that any divergence-free star vector field defined in a closed three-dimensional manifold does not have singularities and moreover it is Anosov (the manifold is uniformly hyperbolic). 
This result was recently generalized in~\cite{F} for a $d$-dimensional closed 
manifold, $d \geq 4$.  
We point out that the proof in~\cite{MBJR} could not be trivially adapted to higher dimensions. Indeed, in dimension $3$, the normal bundle is splitted in two $1$-dimensional subbundles. Consequently, using volume-preserving arguments the authors were able to prove the existence of a dominated splitting for the linear Poincar\'e flow and then the hyperbolicity. 
The main novelties of the proof in~\cite{F} are the use of a new strategy to prove the absence of singularities and the adaptation
of an argument of Ma\~n\'e in~\cite{Ma2} to obtain hyperbolicity from a dominated splitting, which follows easily in dimension $3$. The key ingredient in the proof is the following dichotomy for $C^1$-divergence-free vector fields: 
a periodic orbit of large period either admits a dominated splitting of a prescribed strengh 
or can be turned into a parabolic one by a $C^1$-small perturbation along the orbit.
This dichotomy is a consequence of an adaptation (\cite[Proposition 2.4]{MBJR2}) to the conservative setting of a dichotomy by Bonatti, Gourmelon and Vivier (\cite[Corollary 2.19]{BGV}). 

In the context of Hamiltonian flows, and following the strategy described in~\cite{MBJR}, it was 
obtained in~\cite{BFR} an affirmative answer to the problem of Liao and Ma\~n\'e: any Hamiltonian star system defined on a $4$-dimensional symplectic manifold is Anosov. We remark that the proof makes use 
of some results that are only available in dimension four (see \cite{MBJLD,MBJLD2}). 

In this paper we consider the 
setting of Hamiltonian flows defined on a $2d$-dimensional compact symplectic manifold $(M,\omega)$ ($d \geq 2$). Here, we generalize the results in~\cite{BFR} to higher dimensions and we prove that any Hamiltonian star system 
defined on $2d$-dimensional compact symplectic manifold is Anosov. As a consequence we obtain the proof of the stability conjecture for Hamiltonians.
A key ingredient is a Hamiltonian version of the previously mentioned
dichotomy of Bonatti, Gourmelon and Vivier which will be developed in  Section \S\ref{cocycles}. 

\end{subsection}

\end{section}

\begin{section}{Basic definitions and statement of the results}

\begin{subsection}{Hamiltonians and tangent map structures}
A \textit{Hamiltonian} is a real-valued $C^r$ function on a Riemannian symplectic manifold $M$, $2\leq r\leq \infty$, equipped with a symplectic form $\omega$, whose set is denoted by $C^r(M,\mathbb{R})$. Associated to $H$, we have the Hamiltonian vector field $X_H$ which generates the Hamiltonian flow $X_H^t$. Observe that $H$ is $C^2$ if and only if $X_H$ is $C^1$ and that, since $H$ is continuous and $M$ is compact, $Sing(X_H)\neq\emptyset$, where $Sing(X_H)$ denotes the singularities of $X_H$ or, in other words, the critical points of $H$ or the equilibria of $X_H^t$. Let $\mathcal{R}(H)=M\setminus Sing(X_H)$ stands for the regular points.

A scalar $e\in H(M)\subset \mathbb{R}$ is called an \textit{energy} of $H$. An \textit{energy hypersurface} $\mathcal{E}_{H,e}$ is a connected component of $H^{-1}(\left\{e\right\})$ and it is \textit{regular} if it does not contain singularities. If $H^{-1}(\left\{e\right\})$ is regular, then $H^{-1}(\left\{e\right\})$ is the union of a finite number of energy hypersurfaces. 

\begin{definition}
A \textbf{Hamiltonian system} is a triple $(H,e, \mathcal{E}_{H,e})$, where $H$ is a Ha\-mil\-ton\-ian, $e$ is an energy and $\mathcal{E}_{H,e}$ is a regular connected component of $H^{-1}(\{e\})$.
\end{definition}

Fixing a small neighbourhood $\mathcal{W}$ of a regular $\mathcal{E}_{H,e}$, there exist a small neighbourhood $\mathcal{U}$ of $H$ and $\epsilon>0$ such that, for all $\tilde{H} \in \mathcal{U}$ and $\tilde{e} \in (e-\epsilon,e+\epsilon)$, $\tilde{H}^{-1}(\{\tilde{e}\})\cap \mathcal{W}=\mathcal{E}_{\tilde{H},\tilde{e}}$. We call $\mathcal{E}_{\tilde{H},\tilde{e}}$ the \textit{analytic continuation} of $\mathcal{E}_{H,e}$. 

In the space of Hamiltonian systems we consider the topology generated by a fundamental systems of neighbourhoods.

\begin{definition}
Given a Hamiltonian system $(H,e, \mathcal{E}_{H,e})$ we say that $\mathcal{V}(\mathcal{U},\epsilon)$ is a \textbf{neighbourhood} of $(H,e, \mathcal{E}_{H,e})$ if there exist a small neighbourhood $\mathcal{U}$ of $H$ and $\epsilon>0$ such that for all $\tilde{H} \in \mathcal{U}$ and $\tilde{e} \in (e-\epsilon,e+\epsilon)$ one has that the analytic continuation $\mathcal{E}_{\tilde{H},\tilde{e}}$ of $\mathcal{E}_{H,e}$ is well-defined. 
\end{definition}

For each $x$ in a regular energy hypersurface take the orthogonal splitting $T_xM=\mathbb{R} X_H(x)\oplus N_x$, where $N_x=(\mathbb{R} X_H(x))^\perp$ is the normal fiber at $x$. Consider the automorphism of vector bundles $DX^{t}_{H}\colon  T_{\mathcal{R}}M  \to T_{\mathcal{R}}M$ defined by $DX^{t}_{H}(x,v) = (X^{t}_{H}(x),DX_{H}^{t}(x)\, v)$. Of course that, in general, the subbundle $N_{\mathcal{R}}$ is not $DX_{H}^{t}$-invariant. So we relate to the $DX^{t}_{H}$-invariant quotient space $\widetilde{N}_{\mathcal{R}}=T_{\mathcal{R}}M / \mathbb{R}X_{H}(\mathcal{R})$ with an isomorphism $\phi_{1}\colon N_{\mathcal{R}}\to \widetilde{N}_{\mathcal{R}}$. 
The unique map $P_{H}^{t}\colon N_{\mathcal{R}}\to N_{\mathcal{R}}$ such that $\phi_{1}\circ P_{H}^{t}=DX^{t}_{H}\circ\phi_{1}$ is called the \emph{linear Poincar\'{e} flow} for $H$. Let $\Pi_{x}\colon T_xM\to N_x$ be the canonical orthogonal projection, so the linear Poincar\'e flow  $P^{t}_{H}(x)\colon N_{x}\to N_{X^{t}_{H}(x)}$ is defined by $P^{t}_{H}(x)\, v=\Pi_{X^{t}_{H}(x)}\circ DX^{t}_{H}(x)\, v$.

We now consider 
$$
\NN_x=N_x\cap T_x H^{-1}(e),
$$
where $T_xH^{-1}(e)=\ker dH(x)$ is the tangent space to the energy level set with $e=H(x)$.
Thus, $\mathcal{N}_\mathcal{R}$ is invariant under $P^{t}_{H}$.
So we define the map
$$
\Phi_{H}^{t}\colon\mathcal{N}_{\mathcal{R}}\to\mathcal{N}_{\mathcal{R}},
\qquad
\Phi_{H}^{t}=P^{t}_{H}|_{\mathcal{N}_\mathcal{R}},
$$
called the \emph{transversal linear Poincar\'{e} flow} for $H$ such that
$$
\Phi^{t}_{H}(x)\colon \mathcal{N}_{x}\to \mathcal{N}_{X^{t}_{H}(x)},
\quad
\Phi^{t}_{H}(x)\, v=\Pi_{X^{t}_{H}(x)}\circ DX^{t}_{H}(x)\, v
$$
is a linear symplectomorphism for the symplectic form induced on $\mathcal{N}_\mathcal{R}$ by $\omega$.

\end{subsection}

\begin{subsection}{Invariant splittings and hyperbolicity}

Given $x\in \mathcal{R}(H)$, we say that $x$ is a periodic point if $X^t_H(x)=x$ for some $t$. The smallest $t>0$ is called \emph{period} of $x$ and we denote it by $\pi(x)$. A period point is said to be \emph{hyperbolic} if there exist $\theta\in(0,1)$ and a splitting of the normal subbundle $\mathcal{N}$ along the orbit of $x$,  $\mathcal{N}=E^s\oplus E^u$, such that $\|\Phi_H^t(y)|_{E^s_y}\|<\theta^t$ and $\|(\Phi_H^t(y)|_{E^u_y})^{-1}\|<\theta^t$ for all $y$ in the orbit of $x$. In an analogous way, given a compact and $X^t_H$-invariant set $\Lambda\subset \mathcal{R}(H)$, we say that $\Lambda$ is a \emph{(uniformly) hyperbolic set} if there exist $\theta\in(0,1)$ and a $\Phi^t_H$-invariant splitting of the normal subbundle $\mathcal{N}_\Lambda=E^s_\Lambda\oplus E^u_\Lambda$ such that $\|\Phi_H^t(x)|_{E^s_x}\|<\theta^t$ and $\|(\Phi_H^t(x)|_{E^u_x})^{-1}\|<\theta^t$ for all $x\in\Lambda$. Observe that changing the Riemannian metric the constant of hyperbolicity $\theta$ can be taken equal to $1/2$. Once the metric is fixed, as $\theta$ approximates to $1$ the hyperbolicity gets weaker.

Now, consider a $\Phi_H^t$-invariant splitting $\mathcal{N}=\mathcal{N}^1\oplus\cdots\oplus \mathcal{N}^{k}$ over a compact, $X_H^t$-invariant and regular set $\Lambda$. Assume that, for $1\leq k\leq \dim(M)-2$, all these subbundles have constant dimension. This splitting is \emph{$\ell$-dominated} if there exists ${\ell}>0$ such that, for any $0\leq i<j\leq k$,
\begin{center}
${\|\Phi_H^{\ell}(x)|_{{\mathcal{N}}_x^i}\|}\cdot{\|\Phi_H^{-\ell}(X^{\ell}(x))|_{{\mathcal{N}}_{X^{\ell}(x)}^j}\|}\leq 1/2, \: \: \forall \: x\in \Lambda.$
\end{center}
The $\Phi_H^t$-invariant splitting $\mathcal{N}=\mathcal{N}^u\oplus \mathcal{N}^c \oplus \mathcal{N}^{s}$ over $\Lambda$ is said to be \emph{partially hyperbolic} if there exists ${\ell}>0$ such that,
\begin{enumerate}
\item $\mathcal{N}^u$ is uniformly hyperbolic and expanding with constant of hyperbolicity $1/2$;
\item $\mathcal{N}^s$ is uniformly hyperbolic and contracting with constant of hyperbolicity $1/2$ and
\item $\mathcal{N}^u$ $\ell$-dominates $\mathcal{N}^c$ and $\mathcal{N}^c$ $\ell$-dominates $\mathcal{N}^s$.
\end{enumerate}

We introduce the notion of {\it Hamiltonian star system}.

\begin{definition}\label{hss}
A Hamiltonian system $(H,e, \mathcal{E}_{H,e})$ is a \textbf{Hamiltonian star system} if there exists 
a neighbourhood $\mathcal{V}$ of $(H,e, \mathcal{E}_{H,e})$
such that, for any $(\tilde{H},\tilde{e}, \mathcal{E}_{\tilde{H},\tilde{e}})\in\mathcal{V}$, the correspondent regular energy hypersurface $\mathcal{E}_{\tilde{H},\tilde{e}}$ has all the closed orbits hyperbolic.
We denote by $\mathcal{E}^{\star}_{H,e}$ the regular energy hypersurface with the previous property and by $\mathcal{G}^{2}(M)$ the set of all Hamiltonian star systems defined on a $2d$-dimensional symplectic manifold, $d\geq2$. 
\end{definition}

The next definition states when a Hamiltonian system is \textsl{Anosov}.

\begin{definition}
A Hamiltonian system $(H,e, \mathcal{E}_{H,e})$ is said to be \textbf{Anosov} if $\mathcal{E}_{H,e}$ is uniformly hyperbolic for the Hamiltonian flow $X_H^t$ associated to $H$.
For $d\geq 2$, let $\mathcal{A}(M)$ denote the set of all Anosov Hamiltonian systems, defined on a $2d$-dimensional symplectic manifold.
\end{definition}

\end{subsection}

\begin{subsection}{Statement of the results}

Our main result states that a Hamiltonian star system, defined on a $2d$-dimensional symplectic manifold, is an Anosov Hamiltonian system.

\begin{maintheorem}\label{mainth}
If $(H,e, \mathcal{E}^{\star}_{H,e})\in\mathcal{G}$$^{2}(M)$ then $(H,e, \mathcal{E}^{\star}_{H,e})\in \mathcal{A}(M)$.
\end{maintheorem}

We say that a Hamiltonian system $(H,e, \mathcal{E}_{H,e})$ is \textit{isolated in the boundary of Anosov Hamiltonian systems} if given a neighbourhood  $\mathcal{V}$ of $(H,e, \mathcal{E}_{H,e})$ and $(\tilde{H},\tilde{e}, \mathcal{E}_{\tilde{H},\tilde{e}})\in\mathcal{V}$ the correspondent energy hypersurface $\mathcal{E}_{\tilde{H},\tilde{e}}$ is uniformly hyperbolic but  $\mathcal{E}_{H,e}$ is not. 

As a consequence of Theorem \ref{mainth}, we have the following result.

\begin{maincorollary}\label{BA}
The boundary of $\mathcal{A}(M)$ has no isolated points.
\end{maincorollary}

\begin{definition}
We say that the Hamiltonian system $(H,e,\mathcal{E}_{H,e})$ is \textbf{structurally stable} if there exists a homeomorphism $h_{\tilde{H},\tilde{e}}$ between $\mathcal{E}_{H,e}$ and $\mathcal{E}_{\tilde{H},\tilde{e}}$, preserving orbits and their orientations. Moreover, $h_{\tilde{H},\tilde{e}}$ is continuous on the parameters $\tilde{H}$ and $\tilde{e}$, and converges to $id$ when $\tilde{H}$ $C^2$-converges to $H$ and $\tilde{e}$ converges to $e$. 
\end{definition}

Accordingly with these definitions, we show that structurally stable Hamiltonian systems, defined on a $2d$-dimensional symplectic manifold, are Anosov.

\begin{maintheorem}\label{ssA}
If $(H,e, \mathcal{E}_{H,e})$ is a structurally stable Hamiltonian system, then  $(H,e, \mathcal{E}_{H,e})$ is Anosov.
\end{maintheorem}

\end{subsection}

\end{section}

\begin{section}{Hamiltonian star systems are Anosov - proof of Theorem ~\ref{mainth}}

\begin{subsection}{A dichotomy for Hamiltonian periodic linear differential systems}\label{cocycles}

In this section we intend to contextualize the results in \cite{BGV} for the Hamiltonian scenario. Actually, in \cite{BGV} it is studied the abstract setting of linear bounded cocycles over sets of periodic orbits of large period and it is proved, in brief terms, that a dichotomy between uniform dominated splitting or else one-point spectrum holds (see \cite[Corollary 2.18]{BGV}). In Theorem~\ref{main} we provide a version of this result adapted to Hamiltonians.

We denote by $Sp(2d,\mathbb{R})$ 
($d\geq 1$), the symplectic Lie group of
$2d\times 2d$ matrices $A$ and with real entries satisfying $A^TJA=J$, where
\begin{equation}\label{skew}
J=\begin{pmatrix}0 & -\textbf{1}_{d}\\\textbf{1}_{d} &0\end{pmatrix}
\end{equation}
denotes the skew-symmetric matrix, $\textbf{1}_{d}$ is the $d$-dimensional identity matrix and $A^T$ the transpose matrix of $A$.

Let $\mathbb{R}^{2d}$ be a symplectic vector space equipped with a symplectic form $\omega$. Let $S$ be a two-dimensional subspace of $\mathbb{R}^{2d}$. We denote the $\omega$-\emph{orthogonal complement} of $S$ by $S^{\perp}$ which is defined by those vectors $u\in \mathbb{R}^{2d}$ such that $\omega(u,v) = 0$, for all $v\in S$. Clearly $\dim(S^{\perp})=2d-\dim(S)$. When, for a given subspace $S\subset \mathbb{R}^{2d}$, we have that $\omega|_{S\times S}$ is non-degenerate (say $S^{\perp}\cap S=\{0\}$) then $S$ is said to be a \emph{symplectic subspace}. 

We say that the basis $\{e_1,...,e_{d},e_{\hat{1}},...e_{\hat{d}}\}$ is a \emph{symplectic base} of 
$\mathbb{R}^{2d}$ if $\omega(e_i,e_j)=0$, for all $j\not=\hat{i}$ and $\omega(e_i,e_{\hat{i}})=1$.

Let $\mathfrak {sp}(2d,\mathbb R)$ denote the symplectic $2d$-dimensional Lie algebra of matrices $H$ such that $JH+H^TJ=0$, $\Sigma$ a set of (infinite) periodic orbits and 
$C^{0}(\Sigma, \mathfrak {sp}(2d,\mathbb R))$ denote the space of continuous maps (infinitesimal generators)
with values on the Lie algebra $\mathfrak {sp}(2d,\mathbb R)$ over a Hamiltonian flow $\varphi^t\colon\Sigma\rightarrow\Sigma$. We endow
$C^{0}(\Sigma,\mathfrak {sp}(2d,\mathbb R))$ with the uniform convergence
topology  defined by
$$\|H_0-H_1\|_{0}={\underset{x\in{\Sigma}}{\text{max}}}\|H_0(x)-H_1(x)\|,$$
for any $H_0,H_1\in C^{0}(\Sigma,\mathfrak {sp}(2d,\mathbb R))$.

Given $H \in C^{0}(\Sigma, \mathfrak {sp}(2d,\mathbb R))$, for each $x\in \Sigma$ we
consider the non-autonomous linear differential equation
\begin{equation}\label{lve}
u^\prime(s)\Big|_{s=t}=H(\varphi^{t}(x))\cdot u(t),
\end{equation}
known as \emph{linear variational equation}. Fixing the initial condition $u(0)=\textbf{1}_{2d}$ the unique
solution of~\eqref{lve} is called the \emph{fundamental solution} related to the system $H$. The solution of
~(\ref{lve}) is a linear flow $\Phi_H^t\colon \mathbb
R^{2d}_x\rightarrow  \mathbb R^{2d}_{\varphi^{t}(x)}$, where $\Phi^t_H\in Sp(2d,\mathbb R))$ which may be seen as the skew-product flow

\begin{equation*}
\begin{array}{cccc}
\Phi_H^t\colon  & \Sigma\times\mathbb R^{2d} & \longrightarrow  & \Sigma\times\mathbb R^{2d} \\
& (x,v) & \longrightarrow  & (\varphi^{t}(x),\Phi^{t}_{H}(x)\cdot v).
\end{array}
\end{equation*}

Furthermore, we have the cocycle identity
$\Phi^{t+s}_{H}(x)=\Phi^{s}_{H}(\varphi^{t}(x))\Phi_{H}^{t}(x)$ 
for all $x\in \Sigma$ and $t,s\in\mathbb{R}$. Moreover, $H$ satisfies the differential equation
$H(x)=\frac{d}{dt}\Phi^{t}_{H}(x)|_{t=0}$ for all $x\in \Sigma$. We call $H$ the \emph{infinitesimal generator}
associated to $\Phi_{H}^{t}$. 

Given $H\in \mathfrak {sp}(2d,\mathbb R)$, $\xi>0$ and $P\in \mathfrak {sp}(2d,\mathbb R)$ satisfying $\|P\|_0<\xi$ we say that $H+P$ is a $\xi$-$C^0$-perturbation of $H$. The Hamiltonian dynamics induced by $H+P$ is given by the solution of 
\begin{equation}\label{lve2}
 u^\prime(s)\Big|_{s=t}=(H+P)(\varphi^{t}(x))\cdot u(t).
\end{equation}

We begin by proving a basic perturbation lemma which will be the main tool for obtaining the results from \cite{BGV} to our Hamiltonian context. Roughly, we would like to change a little bit the action of the cocycle on a certain two-dimensional symplectic subspace in time-one. 

\begin{lemma}\label{rot1}
Given $H\in\mathfrak{sp}(2d,\mathbb{R})$ and $\epsilon>0$, there exists $\xi_{0}>0$ (depending on $H$ and $\epsilon$), such that given any $\xi\in(0,\xi_{0})$, any $p\in{\Sigma}$ (with period larger than $1$), any $2$-dimensional symplectic subspace $S_p \subset \mathbb{R}^{2d}_{p}$ and any $R_\xi\in Sp(2d,\mathbb{R})$ which is $\xi$-$C^0$-close to $id$ and $R_\xi |_{W_p}=id$ (where $W_{p}$ is the orthogonal symplectic complement of $S_p$ in $\mathbb{R}^{2d}_{p}$), there exists  $P\in\mathfrak{sp}(2d,\mathbb{R})$ (depending on $\xi$ and $p$) such that:
\begin{enumerate}
\item $\|P\|_0<\epsilon$;
\item $P$ is supported in $\varphi^{t}(p)$ for $t\in[0,1]$;
\item $\Phi^{t}_{H+P}(p)=\Phi^{t}_{H}(p)$ on $W_p$;
\item $\Phi^{1}_{H+P}(p)\cdot v=\Phi^{1}_{H}(p) R_{\xi} \cdot v$, $\forall v \in S_p$.
 \end{enumerate}
\end{lemma}

\begin{proof}
Take $K:=\underset{p\in \Sigma}{\text{max}}\|\Phi_{H}^{\pm t}(p)\|$ for $t\in[0,1]$. We claim that it is sufficient to take $\xi_{0}>0$ such that:
$$\xi_{0}\leq \frac{\epsilon}{4 K^{2}}.$$ 
Let $\alpha\colon \mathbb{R} \rightarrow [0,1]$ be any $C^{\infty}$ function such that $\alpha(t)=0$ for $t\leq 0$, $\alpha(t)=1$ for $t\geq 1$, and $0\leq \alpha^{\prime}(t)\leq 2$, for all $t$. We define the 1-parameter family of symplectic linear maps $\Psi^{t}(p)\colon \mathbb{R}^{2d}_{p} \rightarrow \mathbb{R}^{2d}_{p}$ for $t\in[0,1]$ as follows; we fix two symplectic basis $\{e_1,e_{\hat{1}}\}$ of $S_p$ and $\{e_2,e_3,...,e_d,e_{\hat{2}},e_{\hat{3}},...,e_{\hat{d}}\}$ of $W_p$. 

Take $\xi\in(0,\xi_{0})$ and let $R_\xi\in Sp(2d,\mathbb{R})$ taken $\xi$-$C^0$-close to $id$ and $R_\xi |_{W_p}=id$. Since $S_p \oplus W_p=\mathbb{R}^{2d}_{p}$, given any $u\in \mathbb{R}^{2d}_{p}$ we decompose $u=u_S+u_W$, where $u_S\in S_p$ and $u_W \in W_p$. Let $R_{t}\colon S_p\oplus W_p\rightarrow S_p\oplus W_p$ be an isotopy of symplectic linear maps from $id$ to $R_\xi$ such that:
\begin{enumerate}
\item  [(i)] $R_{t}=\alpha(t)R_\xi+(1-\alpha(t)id)$ and
\item   [(ii)] $R_{t}$ and $R_t^{-1}$ is $\xi$-$C^0$-close to $id$ for any $t\in \mathbb{R}$.
\end{enumerate}

Finally, we consider the 1-parameter family of linear maps $\Psi^{t}(p)\colon \mathbb{R}^{2d}_{p} \rightarrow \mathbb{R}^{2d}_{\varphi^{t}(p)}$ where $\Psi^{t}(p):=\Phi_{H}^{t}(p) R_{t}$. We take time derivatives and we obtain:
\begin{eqnarray*}
(\Psi^{t}(p))^{\prime}&=& (\Phi_{H}^{t}(p))^{\prime}R_{t}+\Phi_{H}^{t}(p)(R_{t})^{\prime}=\\
&=& H(\varphi^{t}(p))\Phi_{H}^{t}(p)R_{t}+\Phi_{H}^{t}(p)(R_{t})^{\prime}=\\
&=& H(\varphi^{t}(p))\Psi^{t}(p)+\Phi_{H}^{t}(p)(R_{t})^{\prime}(\Psi^{t}(p))^{-1}\Psi^{t}(p)=\\
&=& \left[H(\varphi^{t}(p))+P(\varphi^{t}(p))\right]\cdot \Psi^{t}(p).
\end{eqnarray*}
Hence we define the perturbation $P$ by, 
$$P(\varphi^{t}(p))=\Phi_{H}^{t}(p)(R_{t})^{\prime}(R_{t})^{-1}(\Phi_{H}^{t}(p))^{-1}.$$

Let us now show that $P\in\mathfrak{sp}(2d,\mathbb{R})$, that is $JP+P^TJ=0$ holds. Recall the symplectic identities: for any $\Phi\in Sp(2d,\mathbb{R})$; $J^{-1}=J^T=-J$, $\Phi^TJ\Phi=J$ and $\Phi^{-1}=J^{-1}\Phi^TJ$. 
\begin{eqnarray*}
JP+P^TJ&=& J\Phi_{H}^{t}(p)(R_{t})^{\prime}(R_{t})^{-1}(\Phi_{H}^{t}(p))^{-1}+[\Phi_{H}^{t}(p)(R_{t})^{\prime}(R_{t})^{-1}(\Phi_{H}^{t}(p))^{-1}]^TJ\\
&=& (\Phi_{H}^{-t}(p))^TJ(R_{t})^{\prime}(R_{t})^{-1}(\Phi_{H}^{t}(p))^{-1}+(\Phi_{H}^{-t}(p))^T((R_{t})^{-1})^T((R_{t})^{\prime})^T(\Phi_{H}^{t}(p))^TJ\\
&=& (\Phi_{H}^{-t}(p))^TJ(R_{t})^{\prime}(R_{t})^{-1}\Phi_{H}^{-t}(p)+(\Phi_{H}^{-t}(p))^T((R_{t})^{-1})^T((R_{t})^{\prime})^TJ\Phi_{H}^{-t}(p)\\
&=& (\Phi_{H}^{-t}(p))^T[J(R_{t})^{\prime}(R_{t})^{-1}+((R_{t})^{-1})^T((R_{t})^{\prime})^TJ]\Phi_{H}^{-t}(p)\\
&=& (\Phi_{H}^{-t}(p))^TJ[-(R_{t})^{\prime}(R_{t})^{-1}J-J((R_{t})^{-1})^T((R_{t})^{\prime})^T]J\Phi_{H}^{-t}(p)\\
&=& (\Phi_{H}^{-t}(p))^TJ[(R_{t})^{\prime}J^{-1}(R_{t})^T+R_t J^{-1}((R_{t})^{\prime})^T]J\Phi_{H}^{-t}(p)\\
&=& (\Phi_{H}^{-t}(p))^TJ[(R_{t})J^{-1}(R_{t})^T]^{\prime}J\Phi_{H}^{-t}(p)\\
&=& (\Phi_{H}^{-t}(p))^TJ[J^{-1}]^{\prime}J\Phi_{H}^{-t}(p)=0.
\end{eqnarray*}

Now to prove (1) we compute the $C^0$-norm of $P$:
\begin{eqnarray*}
\|P(\varphi^{t}(p))\|_0&=&\|\Phi_{H}^{t}(p)(R_{t})^{\prime}(R_{t})^{-1}(\Phi_{H}^{t}(p))^{-1}\|_0\\
&\leq&K^{2} \|(R_{t})^{\prime}(R_{t})^{-1}\|_0\\
&\leq&K^{2} \|(R_{t})^{\prime}\|_0\|(R_{t})^{-1}\|_0\\
&\leq& 2K^{2} \|(\alpha^\prime(t)R_\xi+(1-\alpha(t)id))^{\prime}\|_0\\
&\leq& 2K^{2} \|\alpha^\prime(t)(R_\xi-id)\|_0\\
&\leq&4K^{2} \|R_\xi-id\|_0\leq 4K^2\xi <\epsilon.
\end{eqnarray*}

Moreover, by our choice of $\alpha$, we have that $\text{Supp}(P)$ is $\varphi^{t}(p)$ for $t\in[0,1]$ and (2) is proved. 
Observe that the perturbed system $H+P$ generates the linear flow $\Phi_{H+P}^{t}(p)$ which is the same as $\Psi^{t}$, hence given $u\in W_{p}$ we have, since $u=u_S+u_W$ (where $u_S=0$):
$$\Phi_{H+P}^{t}(p)\cdot u =\Psi^{t}(p)\cdot u=\Phi_{H}^{t}(p)[R_{t}(u_S)+u_W]=\Phi_{H}^{t}(p)\cdot u_W=\Phi_{H}^{t}(p)\cdot u,$$
and (3) follows.  At last, to prove (4), taking $u\in S_{p}$ we obtain,
\begin{eqnarray*}
\Phi_{H+P}^{1}(p)\cdot u &=&\Psi^{1}(p)\cdot u=\Phi_{H}^{1}(p) R_{1}\cdot u=\Phi_{H}^{1}(p)[R_{\xi}(u_S)+u_W]=\\
&=& \Phi_{H}^{1}(p) R_{\xi}(u_S)=\Phi_{H}^{1}(p) R_{\xi}\cdot u,
\end{eqnarray*}
and Lemma~\ref{rot1} is proved.
\end{proof}

Now, we borrow the arguments in ~\cite{BGV} and use Lemma~\ref{rot1}, when perturbations are needed, in order to obtain the following result which can be seen as the Hamiltonian version of ~\cite[Corollary 2.18]{BGV}. In fact, the perturbations that are used in \cite{BGV} are mainly directional homotheties (contractions and expansions with the same factor) and rotations which are clearly also symplectic.

\begin{theorem}\label{mainBGV}
Given any dimension $2d$ ($d\geq 1$) and any $\epsilon>0$, there
exist $m,n\in\mathbb{N}$ such that any  $H\in\mathfrak{sp}(2d,\mathbb{R})$ over
any periodic orbit $p\in \Sigma$ with period $\pi(p)>n$ satisfies one of the following two assertions:
\begin{enumerate}
\item either $\Phi^t_H(p)$ admits an $m$-dominated splitting; 
\item or there exists an $\epsilon$-$C^0$-perturbation $H+P$ of $H$ such that $\Phi^{\pi(p)}_{H+P}(p)$ has all eigenvalues with modulus equal to $1$.
\end{enumerate}
\end{theorem}

Once we have done the work in the abstract setting of Hamiltonian periodic linear differential systems we would like to consider the $(2d-2)$-linear differential system which is given by the tangent map to the (Hamiltonian) vector field associated to a Hamiltonian defined in a symplectic manifold of dimension $2d$ but ignoring the flow direction and restricted to an energy level. We call this linear differential system the \emph{dynamical} linear differential system. Since we are interested in perturb along closed orbits the framework developed in previous section is the adequated one.

Next, we present a result which is a version of Franks' lemma for Hamiltonians (see~\cite{V}). Roughly, it says that we can realize a Hamiltonian corresponding to a given perturbation of the transversal linear Poincar\'{e} flow. This lemma is the piece that makes possible the connection between abstract linear differential systems and the dynamical one.

\begin{lemma}\label{Frank}  Take $H \in C^2(M,\mathbb{R})$, $\epsilon$, $\tau >0$ and $x\in M$. Then, there exists $\delta>0$ such that for any flowbox $V$ of an injective arc of orbit $X_{H}^{[0,t]}(x)$, $t\geq \tau$, and a transversal symplectic $\delta$-perturbation $F$ of $\Phi_{{H}}^t(x)$, there is $H_0\in C^2(M,\mathbb{R})$ satisfying:
\begin{itemize}
	\item $H_0$ is $\epsilon$-$C^2$-close to $H$;
	\item $\Phi_{{H_0}}^t(x)=F$;
	\item $H=H_0$ on $X_{H}^{[0,t]}(x)\cup (M\backslash V)$.
\end{itemize}
\end{lemma}

Using Lemma~\ref{Frank} and Theorem~\ref{mainBGV} we obtain the following result which will be very useful in the sequel.

\begin{theorem}\label{main}
Let $H\in C^2(M,\mathbb{R})$ and $\mathcal{U}$ be a neighborhood of $H$ in the $C^2$-topology. Then for
any $\epsilon>0$ there are $m,n\in\mathbb{N}$ such that, for any $H_0\in \mathcal{U}$ and for any periodic point
$p$ of period $\pi(p)\geq n$:
\begin{enumerate}
\item either $\Phi^t_{H_0}(p)$ admits an $m$-dominated splitting along the orbit of $p$; 
\item or, for any tubular flowbox neighborhood $\mathcal{T}$ of the orbit of $p$, there exists an $\epsilon$-$C^2$-perturbation $H_1$ coinciding with $H_0$ outside $\mathcal{T}$ and whose transversal linear Poincar\'e flow $\Phi^{\pi(p)}_{H_1}(p)$ has all eigenvalues with modulus equal to $1$.
\end{enumerate}
\end{theorem}
\end{subsection}

\bigskip

\begin{subsection}{Global hyperbolicity}

\begin{lemma}\label{ds}
If $(H,e, \mathcal{E}^{\star}_{H,e})\in\mathcal{G}$$^{2}(M)$, then $\Phi^t_H$ admits a dominated splitting on $\mathcal{E}^{\star}_{H,e}$.
\end{lemma}

\begin{proof}
Consider $(H,e,\mathcal{E}^{\star}_{H,e})\in\mathcal{G}$$^{2}(M)$ and a $C^2$-neighbourhood $\mathcal{V}(\mathcal{U},\epsilon)$ of $(H,e,\mathcal{E}^{\star}_{H,e})$ such that, for any 
$H_0 \in \mathcal{U}$ and any $e_0 \in (e-\epsilon,e+\epsilon)$, the analytic
continuation $\mathcal{E}^{\star}_{H_0,e_0}$ of $\mathcal{E}^{\star}_{H,e}$ also has all the closed orbits hyperbolic, and such that the dicothomy in Theorem~\ref{main} holds. Therefore, by Theorem 
~\ref{main}, there exist positive constants
$m$ and $n$ such that $\Phi^t_{H_0}$ admits an $m$-dominated splitting along the $X^t_{H_0}$-orbit of any periodic point $p$
in $\mathcal{E}^{\star}_{H_0,e_0}$ with period $\pi(p) \geq n$. Observe that, since any periodic point in
$\mathcal{E}^{\star}_{H_0,e_0}$ is hyperbolic, we have the following $X_{H_0}^t$-invariant splitting 
$\mathcal{N}_p=\mathcal{N}^u_p \oplus \mathcal{N}^s_p$ such that any subbundle has constant dimension\footnote{We observe that, in the symplectic context, the index of hyperbolic orbits is always equal to $d$.}.
We claim that this splitting is $m$-dominated for any periodic point $p$ with period $\pi(p) \geq n$. If this claim
is not true, there is a periodic point $q$ with period $\pi(q) \geq n$ such that the angle 
between $\mathcal{N}^u_q$ and $\mathcal{N}^s_q$ is arbitrarily close to $0$ or such that $q$ is wealy hyperbolic.
In these situations, it is straightfoward to see that, applying the Franks' lemma for Hamiltonians (Lemma~\ref{Frank}) several times,
we can $C^2$-perturb $H_0$ in $\mathcal{U}$ in order to have 
$H_1$ such that $q$ is a parabolic closed orbit of $H_1$ in the correspondent energy hypersurface
$\mathcal{E}^{\star}_{H_1,e_1}$. But this is a contradiction, since 
$(H_1,e_1,\mathcal{E}^{\star}_{H_1,e_1}) \in \mathcal{G}$$^{2}(M)$. Therefore, any periodic point $p$ with
period $\pi(p) \geq n$ admits the $m$-dominated splitting $\mathcal{N}_p=\mathcal{N}^u_p \oplus \mathcal{N}^s_p$.

Recall that a dominated splitting can be continuously extended to the closure of a set. Thus, the $m$-dominated 
splitting over the set of periodic points $p$ in $\mathcal{E}^{\star}_{H_0,e_0}$, with
period $\pi(p) \geq n$, can be continuously
extended to its closure. 
Furthermore, we observe that, since $(H_0,e_0,\mathcal{E}^{\star}_{H_0,e_0}) \in \mathcal{G}$$^{2}(M)$, the set of periodic points 
$p$ in  $\mathcal{E}^{\star}_{H_0,e_0}$ with period $\pi(p) < n$  has a finite number of elements. 
Hence, the closure of the set of periodic points $p$ in $\mathcal{E}^{\star}_{H_0,e_0}$ with period $\pi(p) \geq n$   coincides with the set of periodic points in $\mathcal{E}^{\star}_{H_0,e_0}$. 
So, we have just shown that
any Hamiltonian $(H_0,e_0,\mathcal{E}^{\star}_{H_0,e_0})$ in $\mathcal{V}(\mathcal{U},\epsilon)$ admits a dominated splitting on the closure of the 
set of periodic points in $\mathcal{E}^{\star}_{H_0,e_0}$. 

Now, let $x$ be any point in $\mathcal{E}^{\star}_{H,e}$. Clearly, by the Poincar\'e recurrence theorem, $x$ is a non-wandering point. Furthermore, by the Hamiltonian version of the ergodic closing lemma (see
~\cite{Ar}), there exist $H_n \in \mathcal{U}$, $C^2$-converging to $H$, and periodic points $p_n$ of $H_n$ converging to 
$x$. Thus, it follows from above that $x$ can be approximated by periodic points that admit a dominated 
splitting. Since the dominated splitting can be continuously extended to the closure of a set, we obtain that 
$\Phi^t_H$ admits a dominated splitting on $\mathcal{E}^{\star}_{H,e}$.
 
\end{proof}

\begin{remark}\label{isolated}
Observe that the previous lemma remains valid if we assume that $(H,e, \mathcal{E}_{H,e})$
is an isolated point in the boundary of $\mathcal{A}(M)$. In fact, to prove Lemma~\ref{ds}, we use the fact
that $(H,e, \mathcal{E}^{\star}_{H,e})\in\mathcal{G}$$^{2}(M)$ to ensure the existence of a dominated splitting
over a periodic orbit $p$, with arbitrarially large period $\pi$, for a Hamiltonian $H_0$, $C^2$-close to $H$, given by 
Theorem~\ref{main}. Therefore, if we start the proof by assuming that $(H,e, \mathcal{E}_{H,e})$
is an isolated point in the boundary of $\mathcal{A}(M)$, we must obtain the same conclusion, because
any $C^2$-perturbation $H_1$ of $H$ must be Anosov, and so it cannot display a periodic orbit $q$ with period $\pi$ such that $\Phi^{\pi(p)}_{H_1}(p)$ has all eigenvalues with modulus equal to $1$.
\end{remark}

The following auxiliary result asserts that, for a star Hamiltonian $(H,e, \mathcal{E}^{\star}_{H,e})$, any closed orbit is uniformly hyperbolic in the period. This is a crucial step to derive, from Lemma~\ref{ds}, uniform hyperbolicity on $\mathcal{E}^{\star}_{H,e}$.

\begin{lemma}\label{HUP}
Let $(H,e, \mathcal{E}^{\star}_{H,e})\in\mathcal{G}$$^{2}(M)$. There exist a $C^2$-neighbourhood $\mathcal{V}$
of $(H,e, \mathcal{E}^{\star}_{H,e})$  and a constant $\theta \in (0,1)$ such that, for any
$(H_0,e_0, \mathcal{E}^{\star}_{H_0,e_0}) \in \mathcal{V}$, if $p$ is a periodic
point in $\mathcal{E}^{\star}_{H_0,e_0}$ with period $\pi(p)$ and has the hyperbolic splitting
$\mathcal{N}_p=\mathcal{N}^s_p \oplus \mathcal{N}_p^u$ then:
\begin{enumerate}
 \item [$(a)$] $\|\Phi^{\pi(p)}_{H_0} |_{\mathcal{N}_p^s}\| < \theta^{\pi(p)}$ and
\smallskip
 \item [$(b)$] $\|\Phi^{-\pi(p)}_{H_0} |_{\mathcal{N}_p^u}\| < \theta^{\pi(p)}$.
\end{enumerate}
\end{lemma}

\begin{proof}
Given that $(H,e, \mathcal{E}^{\star}_{H,e})\in\mathcal{G}$$^{2}(M)$, there exists a 
$C^2$-neighbourhood $\mathcal{V}(\mathcal{U},\epsilon)$ of $(H,e, \mathcal{E}^{\star}_{H,e})$
such that, for any
$H_0 \in \mathcal{U}$ and any $e_0 \in (e-\epsilon,e+\epsilon)$, the analytic continuation
$\mathcal{E}^{\star}_{H_0,e_0}$ of $\mathcal{E}^{\star}_{H,e}$ also has all the closed orbits hyperbolic. 
This means that for any periodic point $p$ in $\mathcal{E}^{\star}_{H_0,e_0}$, with period $\pi(p)$,
we have that $\mathcal{N}_p=\mathcal{N}^s_p \oplus \mathcal{N}_p^u$ and there is a constant 
$\theta_p \in (0,1)$ such that
\begin{enumerate}
 \item [$(a)$] $\|\Phi^{\pi(p)}_{H_0}(p) |_{\mathcal{N}_p^s}\| < \theta_p^{\pi(p)}$ and
\smallskip
 \item [$(b)$] $\|\Phi^{-\pi(p)}_{H_0}(p) |_{\mathcal{N}_p^u}\| < \theta_p^{\pi(p)}$.
\end{enumerate}
However, we want to prove that, in fact, we can choose $\theta_p$ not depending on $p$. 
Let us prove $(a)$. 
Suppose, by contradiction, that given $\theta=1-2\delta$, with $\delta>0$ small,
there exist $H_0 \in \mathcal{U}$,
$e_0 \in (e-\epsilon,e+\epsilon)$ 
and a periodic point
$p \in \mathcal{E}^{\star}_{H_0,e_0}$, with period $\pi(p)$, hyperbolic by hypothesis, such that
$$(1-2\delta)^{\pi(p)} \leq \|\Phi^{\pi(p)}_{H_0}(p) |_{\mathcal{N}_p^s}\|.$$
Let $A_t$, $0 \leq t \leq \pi(p)$, be the one-parameter family of linear perturbations of $\Phi_{H_0}^t(p)$
given by
$$A_t=\Phi_{H_0}^{t}(p)(1-2\delta)^{-t}.$$
Observe that $\|A_t-\Phi_{H_0}^{t}(p)\|$ can be made arbitrarily close to $0$, taking $\delta$ small enough.
Take $\tilde{\epsilon}>0$ such that any $\tilde{\epsilon}$-$C^2$-perturbation $H_1$ of $H_0$ belongs to $\mathcal{U}$
and take $0<\tau \leq \pi(p)$. 
It follows from Franks' lemma for Hamiltonians (Lemma~\ref{Frank}) that there exists 
$\delta>0$ such that for any flowbox $V$ of an injective arc of orbit $X_{H_0}^{[0,\pi(p)]}(p)$,
there exists $H_1$ $\tilde{\epsilon}$-$C^2$-close to 
$H_0$ coinciding with $H_0$ outside $V$ and such that
$H_1^{\pi(p)}(p)=p$ and $\Phi^{\pi(p)}_{H_1}(p) = A_{\pi(p)}$.
But, by construction, we get that 
$$\|\Phi^{\pi(p)}_{H_1}(p) |_{\mathcal{N}_p^s}\|=
\|\Phi_{H_0}^{\pi(p)}(p)|_{\mathcal{N}_p^s}\|(1-2\delta)^{-\pi(p)}\
=1.$$
This is a contradiction because $p$ is an hyperbolic periodic point of $H_1$.
Then $(a)$ must hold. Item $(b)$ is obtained using a similar argument.
\end{proof}

The following lemma is proved in \cite{BV} and, in brief terms, says
that in the symplectic world, the existence of a dominated splitting
implies partial hyperbolicity.

\begin{lemma}\label{Mane2}
If $\mathcal{N}^u\oplus \mathcal{N}^2$ is a dominated splitting for a
symplectic linear map $\Phi_H^t$, with $\dim \mathcal{N}^u\leq
\dim\mathcal{N}^2$, then $\mathcal{N}^2$ splits invariantly as
$\mathcal{N}^2 = \mathcal{N}^c \oplus \mathcal{N}^s$, with $\dim
\mathcal{N}^s = \dim \mathcal{N}^u$. Furthermore, the splitting
$\mathcal{N}^u \oplus\mathcal{N}^c  \oplus \mathcal{N}^s$ is
dominated, $\mathcal{N}^u$ is uniformly expanding, and $\mathcal{N}^s$
is uniformly contracting. In conclusion, $\mathcal{N}^u
\oplus\mathcal{N}^c  \oplus \mathcal{N}^s$ is partially hyperbolic.
\end{lemma}

Now, by Lemma~\ref{Hyp}, we handle with the last step of the proof of Theorem~\ref{mainth}.

\begin{lemma}\label{Hyp}
If $(H,e, \mathcal{E}^{\star}_{H,e})\in\mathcal{G}$$^{2}(M)$ and
$\mathcal{N}^1\oplus \mathcal{N}^2$ is a dominated splitting, then
this splitting is hyperbolic.
\end{lemma}

\begin{proof}
Since by Lemma~\ref{Mane2} we know that the splitting is partially
hyperbolic it remains to prove that the central subbundle is trivial.
The following arguments are borrowed by the ones of Ma\~n\'e
~\cite{Ma2}. We begin by stating, cf. ~\cite{Ma2}, the following
useful result concerning a dominated splitting $\mathcal{N}^1\oplus
\mathcal{N}^2$.

\begin{claim}\label{mane1}
If $\displaystyle\liminf_{t\rightarrow\infty}\|\Phi_H^t(x)|_{\mathcal{N}^2_x}\|=0$
and $\displaystyle\liminf_{t\rightarrow\infty}\|\Phi_H^{-t}(x)|_{\mathcal{N}^1_x}\|=0$,
for all $x\in \mathcal{E}^{\star}_{H,e}$, then
$\mathcal{E}^{\star}_{H,e}$ is Anosov.
\end{claim}

We shall prove that $\Phi_H^t|_{\mathcal{N}^2}$ is uniformly
contracting on $ \mathcal{E}^{\star}_{H,e}$. That
$\Phi_H^t|_{\mathcal{N}^1}$ is uniformly expanding on $
\mathcal{E}^{\star}_{H,e}$ is analog and we leave it to the reader.
By Claim~ \ref{mane1}, we just have to show that
$$\displaystyle\liminf_{t\rightarrow\infty}\|\Phi_H^t(x)|_{\mathcal{N}^2_x}\|=0,
\:\forall \:x\in  \mathcal{E}^{\star}_{H,e}.$$
By contradiction, assume that there is $x\in
\mathcal{E}^{\star}_{H,e}$ such that
$$\displaystyle\liminf_{t\rightarrow\infty}\|\Phi_H^t(x)|_{\mathcal{N}^2_x}\|>0.$$
Take a subsequence
$t_n\underset{n\rightarrow\infty}{\rightarrow}\infty$ such that
\begin{align}
\displaystyle\lim_{n\rightarrow\infty}\dfrac{1}{t_n}\log\|\Phi_H^{t_n}(x)|_{\mathcal{N}^2_x}\|\geq
0.\label{exp1}
\end{align}
Now, define
\begin{equation*}
\begin{array}{cccc}
\Psi_n\colon  & C^0(\mathcal{E}^{\star}_{H,e}) & \longrightarrow  & \mathbb{R}\\
& f & \longrightarrow  & \frac{1}{s_n}\int_{0}^{s_n} f(X^{s}_H(x))ds
\end{array}
\end{equation*}
where $C^0(\mathcal{E}^{\star}_{H,e})$ stands for the set of
continuous functions on $\mathcal{E}^{\star}_{H,e}$ equipped with the
$C^0$-topology. Take a subsequence of $\Psi_n$ converging to
$\Psi\colon  C^0(\mathcal{E}^{\star}_{H,e})  \rightarrow
\mathbb{R}$. By the Riesz representation theorem, there exists a
$X_H^t$-invariant Borel probability measure $\mu$ defined on
$\mathcal{E}^{\star}_{H,e}$ such that, for any continuous observable
$f$ on $\mathcal{E}^{\star}_{H,e}$ we have,
$$\int_{\mathcal{E}^{\star}_{H,e}}f(x)d\mu(x)=\underset{n\rightarrow\infty}{\lim}
\frac{1}{s_n}\int_{0}^{s_n} f(X^{s}_H(x))ds=\Psi(f).$$
Define a continuous observable $f_H\colon
\mathcal{E}^{\star}_{H,e}\rightarrow \mathbb{R}$ by
$$f_H(x)=\partial_h(\log\|\Phi_H^h(x)|_{\mathcal{N}_x^2}\|)_{h=0}=\underset{h\rightarrow
0}{\lim}\frac{1}{h}\log\|\Phi_H^h(x)|_{\mathcal{N}_x^2}\|.$$

\begin{align}
\int_{ \mathcal{E}^{\star}_{H,e}}f_H(x)
\:d\mu(x)&=\lim_{n\rightarrow+\infty}
\frac{1}{t_n}\int_0^{t_n}f_H(X_H^s(x))\:ds\nonumber\\
&=\lim_{n\rightarrow+\infty} \frac{1}{t_n}\int_0^{t_n}
\partial_h(\log\|\Phi_H^h(X_H^s(x))|_{\mathcal{N}_{X_H^s(x)}^2}\|)_{h=0}\:ds\nonumber\\
&=\lim_{n\rightarrow+\infty} \frac{1}{t_n} \log
\|\Phi_H^{t_n}(x)|_{\mathcal{N}_{x}^2}\|\overset{(\ref{exp1})}{\geq}
0.\nonumber
\end{align}
As a direct consequence of Birkhoff's ergodic theorem we get,
$$\int_{ \mathcal{E}^{\star}_{H,e}}f_H(x) \:d\mu(x)=\int_{
\mathcal{E}^{\star}_{H,e}}\lim_{t\rightarrow+\infty}
\frac{1}{t}\int_0^{t}f_H(X_H^s(x))\:dsd\mu(x)\geq0.$$

Now, let $\Sigma(\mathcal{E}^{\star}_{H,e})$ be the set of points
$x\in \mathcal{E}^{\star}_{H,e}$ such that, for any
$C^2$-neigh\-bour\-hood $\mathcal{U}$ of $H$ and $\delta>0$, there
exist $H_0\in \mathcal{U}$ and a $X^t_{H_0}$-closed orbit $y\in
\mathcal{E}^{\star}_{H,e}$ of period $\pi$ such that $H=H_0$ except on
the $\delta$-neighborhood of the $X_{H_0}^t$-orbit of $y$, and that
$d(X_{H_0}^t(y),X_H^t(x))<\delta$, for $0\leq t\leq \pi$.
By the Hamiltonian version of the ergodic closing lemma (see
~\cite{Ar}), given a $X_H^t$-invariant Borel probability measure
$\mu$, $\mu(\Sigma(\mathcal{E}^{\star}_{H,e}))=1$. So, there is
$x\in\Sigma(\mathcal{E}^{\star}_{H,e})$ such that
\begin{align}\label{expansion}
\lim_{t\rightarrow+\infty}
\frac{1}{t}\int_0^{t}f_H(X^s(x))\:ds=\lim_{t\rightarrow+\infty}
\frac{1}{t}\log\|\Phi_H^{t}(x)|_{\mathcal{N}_{x}^2}\|\geq0.
\end{align}
Fix $\theta\in(0,1)$ given by Lemma~\ref{HUP} and depending on the
neighborhood $\mathcal{U}$ of $H$. Take an arbitrary small $\delta<0$
such that $\log\theta<\delta$. Thus, there is $t_{\delta}$ such that,
for $t\geq t_{\delta}$,
\begin{align}
\frac{1}{t}\log\|\Phi_H^{t}(x)|_{\mathcal{N}_{x}^2}\|\geq\delta.\nonumber
\end{align}
Since $x\in\Sigma(\mathcal{E}^{\star}_{H,e})$, there are
$H_n\in\mathcal{U}$, $C^2$-converging to $H$, and periodic points $p_n$ of $H_n$
with period $\pi_n$. Notice that $\pi_n\rightarrow+\infty$ as
$n\rightarrow\infty$, otherwise, $x$ would be a periodic point of $H$ with period $\pi$ and
the properties of dominated splitting, conservativeness and
(\ref{expansion}) contradict the hypothesis that $H$ has the star
property.

So, assuming that $\pi_n>t_{\delta}$ for every $n$, by the continuity
of the dominated splitting we have that, for $n$ large enough,
$$\|\Phi_{H_n}^{\pi_n}(p_n)|_{\mathcal{N}_{p_n}^2}\|\geq\exp(\delta\pi_n)>\theta^{\pi_n}.$$
But this contradicts (a) in Lemma~\ref{HUP}, because $H_n\in \mathcal{U}$.
So, $\Phi_H^t|_{\mathcal{N}^2}$ is uniformly contracting and the Lemma
is proved.
\end{proof}

As a consequence of Theorem~\ref{mainth} we have the following result. 

\begin{corollary}
 If $(H,e, \mathcal{E}^{\star}_{H,e})\in\mathcal{G}$$^{2}(M)$ then the closure of the set of periodic orbits of $\mathcal{E}^{\star}_{H,e}$ is dense in $\mathcal{E}^{\star}_{H,e}$.
\end{corollary}

\begin{proof} Let $(H,e, \mathcal{E}^{\star}_{H,e})\in\mathcal{G}$$^{2}(M)$.
By Theorem~\ref{mainth}, we have that $(H,e, \mathcal{E}^{\star}_{H,e})$ is an Anosov Hamiltonian system. Now, we use Anosov closing lemma to obtain the conclusion of the corollary.
\end{proof}

We end this section with the proof of Corollary~\ref{BA}.

\begin{proof} (of Corollary~\ref{BA})  By contradiction, assume there exists a Hamiltonian $(H,e, \mathcal{E}_{H,e})$
isolated on the boundary of the set $\mathcal{A}(M)$.
By Remark~\ref{isolated}, $\Phi_H^t$ admits a dominated splitting over $\mathcal{E}_{H,e}$. Therefore,
we just have to follow the proof of Theorem~\ref{mainth}, in order to conclude that 
$(H,e, \mathcal{E}_{H,e}) \in \mathcal{A}(M)$, which is a contradiction.
So, the boundary of the set $\mathcal{A}(M)$ cannot have isolated points.
\end{proof}

\end{subsection}

\end{section}

\begin{section}{Stability conjecture for Hamiltonians - proof of Theorem ~\ref{ssA}}

In this section we prove that structurally stable Hamiltonian systems are Anosov.
Actually, the Hamiltonian version of the structural stability conjecture is
implicitly treated in~\cite[section 6]{N}. See also \cite[Theorem
1.2]{N} and the paragraph before it. However, we cannot find a formal statement and a proof in the literature and for that reason we fill here this gap  although using a different approach from the one implicit in~\cite{N}.

\begin{proof} (of Theorem~\ref{ssA}) 
Let us fix a $C^2$-structurally stable Hamiltonian $(H,e, \mathcal{E}_{H,e})$ and 
choose a $C^2$-neighbourhood $\mathcal{U}$ of $H$ whose elements are topologically equivalent to $H$.
If $H \notin \mathcal{A}(M)=\mathcal{G}$$^{2}(M)$, then there exists $\tilde{H}_0 \in \mathcal{U}$ such that $\tilde{H}_0$ has a non-hyperbolic periodic orbit. Since, by Robinson's version of the Kupka-Smale theorem (see~\cite{Robinson2}), a $C^2$-generic
Hamiltonian has all closed orbits of hyperbolic or elliptic type, there exists $H_0$, close to $\tilde{H}_0$, such that $H_0$ has a $k$-elliptic periodic orbit $p$ of period $\tilde{\pi}$ (recall that a periodic point $p$ of period $\tilde{\pi}$ is $k$-elliptic, 
$1 \leq k \leq d-1$, if $\Phi^{\tilde{\pi}}_{H_0}(p)$
has $2k$ non-real eigenvalues of norm one, and its remaining eigenvalues have norm different from one).

Therefore, there exists a splitting of the normal subbundle $\mathcal{N}^c$ along the orbit of $p$ into
$k$-subspaces $\mathcal{N}^c_j$, $1 \leq j \leq k$, of dimension $2$, such that
$\mathcal{N}^c=\mathcal{N}^c_1 \oplus \ldots \oplus \mathcal{N}^c_k$. 
Let $\theta_1, \ldots, \theta_k \in [0,2\pi[$ be such
that each $\rho_j=\mbox{\rm exp}(\theta_j i)$ is an eigenvalue of $\Phi_{H_0}^{\tilde{\pi}}(p)|_{\mathcal{N}^c_j}$. 
Let $R_{\theta}$ be the rotation matrix of angle $\theta$. 
The Poincar\'e map near $p$, $f_{H_0}$, associated to $\Phi_{H_0}$, is a map from a $(2d-1)$-dimensional manifold
to itself such that when it is restricted to a $(2d-2)$-energy hypersurface it is a local symplectomorphism close 
to $R_{\theta_j}$ in each subspace $\mathcal{N}^c_j$.
Applying Theorem 3 in~\cite{MBJLD3} to $f_{H_0}$, we have that there exists ${H}_1 \in \mathcal{U}$ such that each appropriate restriction of $f_{{H}_1}$ is conjugated to the rotation $R_{\theta_j}$ defined in $\mathcal{N}^c_j$.
We can suppose that each $\theta_j \in \Qq$, i.e., $\theta_j=p_j/q_j$. Otherwise, we slightly perturb each rotation and then apply
~\cite[Theorem 3]{MBJLD3} to obtain a Hamiltonian whose Poincar\'e map restricted to a two-dimensional submanifold $\Sigma_j^c$
is conjugated to a rational rotation, defined in $\mathcal{N}^c_j$, and close to $R_{\theta_j}$, $1 \leq j \leq k$. Take $\ell=\Pi_{j=1}^k q_j$. 
Now, $f_{{H}_1}^{\ell}(q)=q$, for any $q \in \cup_{j=1}^k\Sigma^c_j$. Then, each $q \in \cup_{j=1}^k\Sigma^c_j$ is a periodic point whose period divides $\ell$.
However, as shown by Robinson in~\cite{Robinson2}, $C^2$-generically there are not non-trivial resonance relations.
In particular, $C^2$-generically the periodic orbits are isolated. So, ${H}_1$ 
must be conjugated to a Hamiltonian which has only a finite number of closed orbits with period is limited by $\max\{\tilde{\pi},\ell\}$. 
As, by the definition of structural stability, the conjugation is close to the $id$, this leads to a contradiction.
\end{proof}

\end{section}

\section*{Acknowledgements}

MB was partially supported by National Funds through FCT - ``Funda\c{c}\~{a}o para a Ci\^{e}ncia e a Tecnologia", project PEst-OE/MAT/UI0212/2011.

MJT was partially financed by FEDER Funds through ``Programa Operacional Factores de Competitividade - COMPETE'' and by Portuguese Funds through FCT, within the Project PEst-C/MAT/UI0013/2011.

JR was partially funded by the European Regional Development Fund through the programme COMPETE and by the Portuguese Government through the FCT under the project PEst - C/MAT/UI0144/2011. 

MB and JR partially supported by the FCT, project PTDC/MAT/099493/2008.

\end{document}